\documentclass[12pt,leqno,oneside,letterpaper]{amsart}
\usepackage[dvips,text={6.5truein,9truein},left=1truein,top=1truein]{geometry}
\usepackage{amssymb,amsmath,amscd,enumerate}
\usepackage[pdftex]{graphicx}
\usepackage{url}

\usepackage[colorlinks,linkcolor=blue,citecolor=blue,pdfstartview=FitH]{hyperref}
\usepackage{color}
\usepackage{tikz}

\title[Local maxima of maximal injectivity radius]{The local maxima of maximal injectivity radius among hyperbolic surfaces}

\author{Jason DeBlois}
\address{Department of Mathematics\\University of Pittsburgh\\\newline
  Pittsburgh, PA 15260}
\email{jdeblois@pitt.edu}
\urladdr{http://www.pitt.edu/~jdeblois}

\newcommand\co{\colon\thinspace}

\theoremstyle{definition}
\newtheorem{para}{}[section]

\newtheorem{definition}[para]{Definition}

\newtheorem{claim}[equation]{Claim}

\theoremstyle{plain}
\newtheorem{theorem}[para]{Theorem}
\newtheorem{lemma}[para]{Lemma}
\newtheorem{proposition}[para]{Proposition}
\newtheorem{corollary}[para]{Corollary}

\numberwithin{equation}{para}
\numberwithin{figure}{section}

\newcommand\calAC{\mathcal{AC}}
\newcommand\calBC{\mathcal{BC}}
\newcommand\calHC{\mathcal{HC}}

\newcommand\calc{{\mathcal C}}
\newcommand\cale{\mathcal{E}}

\newcommand\cals{\mathcal{S}}
\newcommand\calt{{\mathcal T}}

\newcommand\bd{\mathbf{d}}

\newcommand\bx{\mathbf{x}}

\newcommand\injrad{\mathit{injrad}}

\newcommand\maxi{\mathit{max}}
\newcommand\dist{d}

\begin{document}

\begin{abstract}
The function on the Teichm\"uller space of complete, orientable, finite-area hyperbolic surfaces of a fixed topological type that assigns to a hyperbolic surface its maximal injectivity radius has no local maxima that are not global maxima.\end{abstract}

\maketitle

Let $\mathfrak{T}_{g,n}$ be the Teichm\"uller space of complete, orientable, finite-area hyperbolic surfaces of genus $g$ with $n$ cusps.  In this paper we begin to analyze the function $\maxi\co\mathfrak{T}_{g,n}\to\mathbb{R}^+$ that assigns to $S\in\mathfrak{T}_{g,n}$ its maximal injectivity radius.  The \textit{injectivity radius} of $S$ at $x$, $\injrad_x(S)$, is half the length of the shortest non-constant geodesic arc in $S$ with both endpoints at $x$.  It is not hard to see that $\injrad_x(S)$ varies continuously with $x$ and approaches $0$ in the cusps of $S$, so it attains a maximum on any fixed finite-area hyperbolic surface $S$.

Our main theorem characterizes local maxima of $\maxi$ on $\mathfrak{T}_{g,n}$:

\begin{theorem}\label{only max}  For $S\in\mathfrak{T}_{g,n}$, the function $\maxi$ attains a local maximum at $S$ if and only if for each $x\in S$ such that $\injrad_x(S) = \maxi(S)$, each edge of the Delaunay tessellation of $(S,x)$ has length $2\injrad_x(S)$ and each face is a triangle or monogon.\end{theorem}

Here for a hyperbolic surface $S$ with locally isometric universal cover $\pi\co\mathbb{H}^2\to S$, and $x\in S$, the \textit{Delaunay tessellation of $(S,x)$} is the projection to $S$ of the Delaunay tessellation of $\pi^{-1}(x)\subset\mathbb{H}^2$, as defined by an empty circumcircles condition (see Section \ref{Delaunay} below).  In particular, a \textit{monogon} is the projection to $S$ of the convex hull of a $P$-orbit in $\pi^{-1}(x)$, for a maximal parabolic subgroup $P$ of $\pi_1 S$ acting on $\mathbb{H}^2$ by covering transformations.

Theorem 5.11 of the author's previous paper \cite{DeB_Voronoi} characterized the global maxima of $\maxi$ by a condition equivalent to that of Theorem \ref{only max}, extending work of Bavard \cite{Bavard}.  We thus have:

\begin{corollary}\label{local to global}  All local maxima of $\maxi$ on $\mathfrak{T}_{g,n}$ are global maxima.\end{corollary}

This contrasts the behavior of $\mathit{syst}$, the function on $\mathfrak{T}_{g,n}$ that records the \textit{systole}, ie.~shortest geodesic, length of hyperbolic surfaces: P.~Schmutz Schaller proved in \cite{Schmutz} that for many $g$ and $n$, $\mathit{syst}$ has local maxima on $\mathfrak{T}_{g,n}$ that are not global maxima.  Comparing with $\mathit{syst}$, which is well-studied, is one motivation for studying $\maxi$.  (Note that for a closed hyperbolic surface $S$, $\mathit{syst}(S)$ is twice the \textit{minimal} injectivity radius of $S$.)

The referee has sketched a direct argument to show that $\maxi$ attains a global maximum on $\mathfrak{T}_{g,n}$.  (This is also sketched in the preprint \cite{Gendulphe}, and I prove a somewhat more general fact as Proposition 4.3 of \cite{DeB_many}.)  Together with this observation, Theorem \ref{only max} gives an alternative proof of Theorem 5.11 of \cite{DeB_Voronoi}, which is not completely independent of the results of \cite{DeB_Voronoi} but uses only some early results from Sections 1 and Section 2.1 there.

We prove Theorem \ref{only max} by describing explicit, injectivity radius-increasing deformations of pointed surfaces $(S,x)$ that do not satisfy its criterion.  The deformations are produced by changing finite edge lengths of a decomposition $\calt$ of $S$ into compact and horocyclic ideal triangles, with vertex set $x$.  In Section \ref{deform} we introduce a space $\mathfrak{D}(S,\calt)$ parametrizing such deformations.  Proposition \ref{parametrize} there shows that the natural map $\mathfrak{D}(S,\calt)\to \mathfrak{T}_{g,n}$ is continuous.  Proposition \ref{Delaunay control} gives a simple description of $\maxi$ near $\bd\in\mathfrak{D}(S,\calt)$ in terms of the edge lengths, assuming that all shortest arcs based at $x$ are edges of $\calt$.

%The first main result of Section \ref{geom} is Proposition \ref{converge}, describing a convergence property of the developing maps associated to $\bd\in\mathfrak{D}(S,\calt)$.  We use this property in Proposition \ref{Delaunay control} to

By Lemma \ref{injrad edge}, all such arcs are Delaunay edges.  Section \ref{Delaunay} introduces the Delaunay tessellation of $(S,x)$, following the author's prior paper \cite{DeB_Delaunay}, and describes its relevance to this paper.  In particular, we prove Theorem \ref{only max} using triangulations obtained by subdividing the Delaunay tessellation, see Lemma \ref{compatible}.

Section \ref{the proof, man} is devoted to the proof of Theorem \ref{only max}.  Proposition \ref{longer than} reduces it to the case that every Delaunay edge has length $2\mathit{injrad}_x(S)$.  We believe it has more to say about the critical set of $\maxi$ (properly interpreted, since $\maxi$ is not smooth) and hope in future work to more deeply understand this set.  Theorem \ref{only max} is then proved by showing that complicated Delaunay cells can be broken apart by injectivity radius-increasing deformations.  The arguments in this section use some basic observations from \cite{DeB_Voronoi} and, like the results there, exploit what you might call the ``calculus of cyclic polygons'' laid out in \cite{DeB_cyclic_geom}. 

\subsection*{Acknowledgements} We were originally motivated and in part inspired by a private communication from Ian Agol, where he sketched a proof of \cite[Theorem 5.11]{DeB_Voronoi} using deformations through hyperbolic cone surfaces and conjectured Corollary \ref{local to global}.  We thank Ian for his interest and ideas.  We are also grateful to the referee for helpful comments which have significantly improved the paper, in particular for the simplified proof of Proposition \ref{parametrize}.

After this paper was first submitted, M.~Gendulphe posted the preprint \cite{Gendulphe} which proves Theorem \ref{only max} by a different method.

%%%%%%%%%%%%%%%%%%%%%%%%%%%
\section{Deformations via triangulations}\label{deform}
%%%%%%%%%%%%%%%%%%%%%%%%%%%

Let us begin this section by fixing a complete, oriented, finite-area hyperbolic surface $S$ of genus $g$ with $n$ cusps and a decomposition $\mathcal{T}$ of $S$ into compact and horocyclic ideal hyperbolic triangles.  (Here a \textit{horocyclic ideal triangle} has two vertices on a horocycle of $\mathbb{H}^2$ and an ideal vertex at the horocycle's ideal point.)  We will call the pair $(S,\calt)$ a \textit{triangulated surface} for short.  The main results of this section are Propositions \ref{parametrize} and \ref{Delaunay control}.% asserts that deforming finite edge lengths of $\calt$ determines deformations of $S$ in $\mathfrak{T}_{g,n}$.

We first define a space $\mathfrak{D}(S,\calt)$ of possible deformations of the edge lengths of $\calt$.

\begin{definition}\label{triangle def}  Suppose $(S,\calt)$ is a complete, oriented, triangulated hyperbolic surface of finite area, and enumerate the faces of $\calt$ as $F_1,\hdots,F_k$ and the compact edges as $e_1,\hdots,e_l$.  Let $U\subset(0,\infty)^l$ be the set of $\bd = (d_0,\hdots,d_l)$ such that $d_{j_1} < d_{j_2}+d_{j_3}$ for any $j_1,j_2,j_3$ such that $e_{j_1},e_{j_2}$ and $e_{j_3}$ are the distinct edges of some $F_i$.  For each vertex $x$ of $\calt$ define:
$$ A_x(\bd) = \sum_i \cos^{-1}\left(\frac{\cosh d_{j_1}\cosh d_{j_2} - \cosh d_{j_3}}{\sinh d_{j_1}\sinh d_{j_2}}\right) + \sum_k \sin^{-1}\left(\frac{1}{\cosh(d_{j_k}/2)}\right) $$
This sum is taken over all $i$ such that $x$ is in the compact triangle $F_i$, where $e_{j_1}$ and $e_{j_2}$ are the edges of $\calt$ containing $x$, and $e_{j_3}$ is the edge of $F_i$ opposite $x$, and all $k$ such that $x$ is in the horocyclic triangle $F_k$ with finite side length $d_{j_k}$.  We then take:
$$ \mathfrak{D}(S,\calt) = \{\bd\in U\,|\, A_x(\bd) = 2\pi\ \mbox{for each vertex $x$ of $\calt$}\} $$
We call this the set of \textit{deformations of $(S,\calt)$}.\end{definition}

\begin{lemma}\label{hyperbolize}  Suppose $(S,\calt)$ is a complete, oriented, triangulated hyperbolic surface of finite area.  For each face $F_i$ of $\calt$ and each $\bd=(d_1,\hdots,d_l)\in\mathfrak{D}(S,\calt)$ let $F_i(\bd)$ be the compact hyperbolic triangle with edge lengths $d_{j_1}$, $d_{j_2}$ and $d_{j_3}$, if $F_i$ is compact with edges $e_{j_1}$, $e_{j_2}$ and $e_{j_3}$; or otherwise the horocyclic ideal triangle with finite edge length $d_{j_1}$, where $F_i$ has compact edge $e_{j_1}$.  The triangulated polyhedral complex $(S(\bd),\calt(\bd))$ obtained by identifying edges of the $F_i(\bd)$ in pairs corresponding to edges of $\calt$ inherits a complete hyperbolic structure from the $F_i(\bd)$, and it has a homeomorphism to $(S,\calt)$ taking $F_i(\bd)$ to $F_i$ for each $i$.\end{lemma}

\begin{proof}  Let us recall some standard facts.  Below, for a compact hyperbolic triangle with sides of length $a$, $b$ and $c$ let $\alpha$ be the interior angle opposite the side of length $a$.  Let $\delta$ be the interior angle at either endpoint of the finite edge, with length $d$, of a horocyclic ideal triangle.  Then:
\begin{align}\label{angles}
  & \alpha = \cos^{-1}\left(\frac{\cosh b\cosh c-\cosh a}{\sinh b\sinh c}\right)\in (0,\pi) &
  & \delta = \sin^{-1}\left(\frac{1}{\cosh(d/2)}\right) \in (0,\pi/2) \end{align}
The left-hand equation is the hyperbolic law of cosines (see eg.~\cite[Theorem 3.5.3]{Ratcliffe}).  The right can be proved by an explicit calculation in, say, the upper half-plane model $\mathbb{R}\times(0,\infty)$ for $\mathbb{H}^2$, placing the horocycle at $\mathbb{R}\times\{1\}$ and using the fact that the Euclidean and hyperbolic distances $\ell$ and $d$, respectively, between points on it satisfy $\ell/2=\sinh(d/2)$.

Now with the faces and compact edges of $\calt$ enumerated as in Definition \ref{triangle def}, for $\bd\in\mathfrak{D}(S,\calt)$ and $1\leq i\leq k$ let $F_i(\bd)$ be as described in the statement.  By construction and the formulas of (\ref{angles}), $A_x(\bd) = 2\pi$ is the sum of all vertex angles of the $F_i(\bd)$.  Note also that each edge of $\calt$ is contained in two faces, and again by construction if this edge is compact then the corresponding edges of the $F_i(\bd)$ have the same length.

For $(S(\bd),\calt(\bd))$ as described in the statement, there is clearly a triangulation-preserving homeomorphism $(S(\bd),\calt(\bd))\to(S,\calt)$.  Moreover, choosing a disjoint collection of representatives of the $F_i(\bd)$ in $\mathbb{H}^2$ it is not hard to arrange for the pairing of edges to be realized by an $\mathrm{Isom}^+(\mathbb{H}^2)$-\textit{side pairing} in the sense of \cite[\S 9.2]{Ratcliffe}.  Theorem 9.2.2 of \cite{Ratcliffe} then implies that $S(\bd)$ inherits a hyperbolic structure from the $F_i(\bd)$.  The key requirement for this result, that the side-pairing is \textit{proper}, obtains from the fact that the angle sum $A_x(\bd)$ around the vertex $x$ is $2\pi$.

We further claim that the hyperbolic structure on $S(\bd)$ is complete, see \cite[Theorem 8.5.9]{Ratcliffe}.  This follows from the stipulation in Definition \ref{triangle def} that non-compact faces of $\calt$ are horocyclic ideal triangles, since an isometry that takes an infinite edge of one horocyclic ideal triangle to an infinite edge of another identifies the horocycles containing their vertices.  For any such face with ideal vertex $v$, we thus have $d(v)=0$, where $d(v)$ is the ``gluing invariant'' of \cite[\S 9.8]{Ratcliffe}, so $S$ is complete by Theorem 9.8.5 of \cite{Ratcliffe} (cf.~\cite[Prop.~3.4.18]{Th_book}).   \end{proof}

We next relate the deformation space $\mathfrak{D}(S,\calt)$ to the Teichm\"uller space $\mathfrak{T}_{g,n}$ of hyperbolic surfaces with the topological type of $S$, endowed with its standard topology (see eg. \cite{FaMa}).  Here we will regard the hyperbolic surfaces $S(\bd)$ from Lemma \ref{hyperbolize} as marked by the homeomorphism from $S$ described there.

\begin{proposition}\label{parametrize} For a complete, oriented, triangulated hyperbolic surface $(S,\calt)$ of finite area, with genus $g$ and $n$ cusps, the map $\mathfrak{D}(S,\calt)\to\mathfrak{T}_{g,n}$ given by $\bd\mapsto S(\bd)$ is continuous.\end{proposition}

\begin{proof}  We will show that for any essential simple closed curve $\gamma$ on $S$, the function $\bd\mapsto \ell_{S(\bd)}(\gamma)$ that measures the geodesic length of $\gamma$ in $S(\bd)$ is continuous.  It then follows from standard results, eg. the ``$9g-9$ theorem'' \cite[Theorem 10.7]{FaMa}, that $\bd\mapsto S(\bd)$ is continuous.

Fix $\bd\in\mathfrak{D}(S,\calt)$, and refer by $\gamma$ to an oriented geodesic representative of $\gamma$ on $S(\bd)$.  
For $\bd'$ near to $\bd$, we now construct a piecewise-geodesic $\gamma'$ on $S(\bd')$ which will be evidently isotopic to the image of $\gamma$ under the homeomorphism $S(\bd)\to S(\bd')$ described in Lemma \ref{hyperbolize}.  We will then show that the length of $\gamma'$ exceeds that of $\gamma$ by no more than some $\epsilon$ depending on $\delta = \max\{|\delta_i|\}$,  where $\delta_i = |d_i - d_i'|$ for each $i$, which approaches $0$ as $\delta\to 0$.

Partition $\gamma$ into arcs $\gamma_0,\hdots,\gamma_{k-1}$ with disjoint interiors such that $\gamma_j$ is adjacent to $\gamma_{j+1}$ for each $j$, and each $\gamma_j$ is either an edge of or (the generic case) properly embedded in a triangle of $\calt$.  For each $j$ we construct a geodesic arc $\gamma_j'$ in $S(\bd')$ that lies in the same triangle(s) of $\calt$ as $\gamma_j$, as follows.  For each endpoint $x$ of $\gamma_j$ that is a vertex of $\calt$, let the corresponding endpoint $x'$ of $\gamma_j'$ be the same vertex; if $x$ lies in the interior of a compact edge $e_{i_j}$ of $\calt$, then with $\delta_{i_j}$ as above let $x'$ lie on $e_{i_j}$ in $S(\bd')$ with $|d(x',v) - d(x,v)| = \delta_{i_j}/2$ for each vertex $v$ of $e_{i_j}$; and if $x$ is in the interior of a non-compact edge of $\calt$ then let $x'$ lie on the same edge in $S(\bd')$, at the same distance from its (finite) vertex.  Now let $\gamma' = \gamma_1'\cup\hdots\gamma_{k-1}'$.

For any fixed $j$, we will show that $|\ell_j'-\ell_j|\to 0$ as $\delta\to 0$, where $\ell_j$ and $\ell_j'$ are the respective lengths of $\gamma_j$ and $\gamma_j'$ in $S(\bd)$ and $S(\bd')$.  If $\gamma_j$ is an edge of $\calt$ this is obvious, so let us assume it is not. Then $\gamma_j$ cuts the triangle $T$ of $\calt$ containing it into two pieces, at least one of which is a triangle. If the vertex $v$ of the sub-triangle of $T$ that does not lie in $\gamma$ is finite, and the interior angle there is $\theta_j$, then by the hyperbolic law of cosines $\gamma_j$ has length $\ell_j$ given by
$$ \cosh\ell_j = \cosh d(x_j,v) \cosh d(x_{j+1},v) - \sinh d(x_j,v) \sinh d(x_{j+1},v) \cos\theta_j. $$
Here $x_j$ and $x_{j+1}$ are the endpoints of $\gamma_j$, and $\theta_j$ is given in terms of $\bd$ by the left side of the formula (\ref{angles}).  For $\ell_j'$ we substitute $x_j'$, $x_{j+1}'$ and $\theta_j'$ above.  It is clear from this formula that $|\ell_j' - \ell_j| \to 0$ as $\delta\to 0$.

If the vertex $v$ described above is ideal then we claim that the length $\ell_j$ of $\gamma_j$ satisfies
\[ \cosh \ell_j = \frac{e^{-d(x_j,v_j)-d(x_{j+1},v_{j+1})}}{2}\left(4\sinh^2(d_{i_j}/2) + e^{2d(x_j,v_j)}+e^{2d(x_{j+1},v_{j+1})}\right), \]
where $v_j$ is the finite vertex of the edge $e$ of $T$ containing $x_j$, $v_{j+1}$ is the finite vertex of the edge containing $x_{j+1}$, and $d_{i_j}$ is the length of the compact edge of $T$. 

This follows from explicit computations in the upper half-plane model.  Applying an isometry, we may take $T$ inscribed in the horocycle $C=\mathbb{R}+i$, with $v_j = i$ and $v_{j+1} = \theta+i$, where $\theta = 2\sinh(d_{i_j}/2)$ is the distance from $v_j$ to $v_{j+1}$ along $C$.  Then $x_j = iy_0$ and $x_{j+1} = \theta+iy_1$ for $y_0,y_1>1$ satisfying $e^{d(x_j,v_j)} =y_0$ and $e^{d(x_{j+1},v_{j+1})} = y_1$.  Theorem 1.2.6(ii) of \cite{Katok} now proves the claim, giving:
\[ \cosh \ell = 1 + \frac{\theta^2 + (y_0-y_1)^2}{2y_0y_1} = \frac{\theta^2 + y_0^2+y_1^2}{2y_0y_1}  \]

To compute the length $\ell_j'$ of the corresponding arc $\gamma_j'$ we simply replace $d_{i_j}$ by $d_{i_j}'$ above, where $d_{i_j}'$ is the length in $S(\bd')$ of the compact edge of $T$.  Convergence of $\ell_j'$ to $\ell_j$ thus follows as in the previous case.

Since the length of $\gamma$ is $\sum\ell_i$, and the geodesic length of $\gamma'$ is at most $\sum\ell_i'$, this exceeds the length of $\gamma$ by no more than some $\epsilon = \epsilon(\delta)$, which approaches $0$ as $\delta\to 0$.  From the formulas above we see that the dependence of $\epsilon$ on $\delta$ is uniform on compact subsets of the open set $U$ of Definition \ref{triangle def}.  Therefore we can apply the same argument with the roles of $\gamma$ and $\gamma'$ reversed to show that the length of $\gamma$ exceeds the geodesic length of $\gamma'$ by no more than some $\epsilon' = \epsilon(\delta)$ which also approaches $0$ as $\delta\to 0$.  Continuity follows.
\end{proof}

\begin{proposition}\label{Delaunay control}  Suppose $(S,\calt)$ is a complete, triangulated hyperbolic surface of finite area with vertex set $\{x\}$ such that the entire collection of geodesic arcs of length $2\mathit{injrad}_x(S)$ based at $x$ is a set $e_{j_1},\hdots,e_{j_n}$ of edges of $\calt$.  Then there is a neighborhood $V$ in $\mathfrak{D}(S,\calt)$ of the edge length collection of $\calt$ such that for any $\bd\in V$, if $x_{\bd}$ is the vertex of the triangulated hyperbolic surface $(S(\bd),\calt(\bd))$ of Lemma \ref{hyperbolize} then $\mathit{injrad}_{x(\bd)} S(\bd) = \frac{1}{2}\min\{d_{j_i}\}_{i=1}^n$.\end{proposition}

\begin{proof}  The collection of geodesic arcs in $S$ based at $x$ is in 1-1 correspondence with those in $\mathbb{H}^2$ joining $\tilde{x}$ to other points of $\pi^{-1}(x)$, where $\pi\co\mathbb{H}^2\to S$ is a locally isometric universal cover and $\tilde{x}$ is a fixed element of $\pi^{-1}(x)$.  Fix some $R>0$ that is slightly larger than $2\mathit{injrad}_x(S)$, and let $P\subset\mathbb{H}^2$ be the union of lifts of triangles of $\calt$ that intersect the closed ball of radius $R$ about $\tilde{x}$.  This is a finite union since the lifted triangulation is locally finite.

For $\bd'$ near the edge length collection $\bd$ of $S$ in $\mathfrak{D}(S,\calt)$, let $P'$ be the corresponding union of triangles lifted from $(S(\bd'),\calt(\bd'))$.  That is, fix a locally isometric universal cover $\pi'\co\mathbb{H}^2\to S(\bd')$ and some $\tilde{x}'\in(\pi')^{-1}(x')$, where $x'$ is the vertex of $\calt(\bd')$, and let $P'$ be the image of $P$ under the lift that takes $\tilde{x}$ to $\tilde{x}'$ of the marking $S\to S(\bd')$.  (Recall that this map takes triangles to triangles.)

For each geodesic arc $\gamma$ in $S$ based at $x$ and any fixed $\epsilon>0$, arguing as in the proof of Proposition \ref{parametrize} shows that $\bd'$ can be chosen near enough to $\bd$ that the geodesic arc in $S(\bd')$ based at $x'$ in the based homotopy class of the image of $\gamma$ has length less than $\epsilon$ away from the length of $\gamma$.  In particular, there exists $\delta>0$ such that if $\max\{|d_i-d_i'|\}<\delta$ then for each vertex $v$ of $P$ at distance greater than $2\mathit{injrad}_x(S)$ from $\tilde{x}$, the image of $v$ has distance to $\tilde{x}'$ greater than $i(\bd')\doteq\min\{\ell_{j_1}',\hdots,\ell_{j_n}'\}$, where $\ell_{j_i'}$ is the length of $e_{j_i}$ in $S(\bd')$ for each $i$.

We now consider geodesic arcs in $S$ based at $x$ whose lifts based at $\tilde{x}$ exit $P$.  For each such arc the analogous fact holds for its correspondent in $S(\bd')$.  We will thus complete the proposition's proof by showing that for small enough $\delta>0$, the closest point to $\tilde{x}'$ on each edge in the frontier of $P'$ is at distance greater than $i(\bd')$ (defined above) from it, whence $i(\bd')$ is twice the injectivity radius of $S(\bd')$ at $x'$.

For a compact edge $e$ in the frontier of $P$, let $T$ be the triangle with one edge at $e$ and opposite vertex $\tilde{x}$. The closest point of $e$ to $\tilde{x}$ is in its interior if and only if the angles of $T$ are less than $\pi/2$ at each endpoint of $e$.  In this case the geodesic arc from $\tilde{x}$ to its closest point on $e$ intersects $e$ at right angles, and by the hyperbolic law of sines the distance $h$ from $\tilde{x}$ to $e$ satisfies $\sinh h  = \sinh \ell \sin\theta$.  Here $\ell$ is the distance from $\tilde{x}$ to an endpoint $v$ of $e$, and $\theta$ is the interior angle of $T$ at $v$.

If $T'$ is the corresponding triangle in $P'$ then for $\bd'$ near $\bd$, each edge length of $T'$ is near the corresponding edge length of $T$, as we have already remarked, and it follows from the hyperbolic law of cosines that the same holds for the angles of $T$ and $T'$.  In particular, if the closest point of $e$ to $\tilde{x}$ is in the interior of $e$ then for $\bd'$ near enough to $\bd$, the closest point of the corresponding edge $e'$ to $\tilde{x}'$ is also in its interior, and by the hyperbolic law of sines the distance $h'$ from $\tilde{x}'$ to $e'$ approaches $h$ as $\bd'\to\bd$.  The remaining case is straightforward.

For a non-compact edge $e$ in the frontier of $P$, if the nearest point of $e$ to $\tilde{x}$ is in its interior then we again use the formula $\sinh h = \sinh \ell\sin\phi$, where now $\ell$ is the length of the geodesic arc from $\tilde{x}$ to the finite endpoint $v$ of $e$, and $\phi$ is the angle from this arc to $e$.  It follows as before that $\ell'\to\ell$ as $\bd'\to\bd$.  To see that the corresponding angle $\phi'$ approaches to $\phi$ as $\bd'\to\bd$ we note that $\phi = \theta+\delta$, where $\delta$ is the interior angle at $v$ of the horocyclic triangle $T_0$ in $P$ containing $e$, and $\theta$ is the interior angle at $v$ of the triangle determined by $\tilde{x}$ and the finite side $f$ of $T_0$.  The corresponding angles $\delta'\to\delta$ and $\theta'\to\theta$, so $\phi'\to\phi$ as $\bd'\to\bd$.\end{proof}

%%%%%%%%%%%%%%%%%%%%%%%%%%
\section{The Delaunay tessellation}\label{Delaunay}
%%%%%%%%%%%%%%%%%%%%%%%%%%

In this section, for a hyperbolic surface $S$ and $x\in S$ we define the Delaunay tessellation of $(S,x)$ (Definition \ref{Delaunay dfn} below) by projecting Delaunay cells of $\pi^{-1}(x)$, where $\pi\co\mathbb{H}^2\to S$ is the universal cover.  Here the Delaunay tessellation of a locally finite, lattice-invariant subset $\widetilde\cals\subset\mathbb{H}^2$, in the sense of \cite[Theorem 5.1]{DeB_Voronoi} (which itself is the specialization to two dimensions of \cite[Theorem 6.23]{DeB_Delaunay}), is characterized by the \textit{empty circumcircles condition}:

\begin{quote}For each circle or horocycle $H$ of $\mathbb{H}^2$ that intersects $\widetilde\cals$ and bounds a disk or horoball $B$ with $B\cap\widetilde\cals=H\cap\widetilde\cals$, the closed convex hull of $H\cap\widetilde\cals$ in $\mathbb{H}^2$ is a Delaunay cell.  Each Delaunay cell has this form.\end{quote}

In proving Theorem \ref{only max} we will use triangulations \textit{compatible} with the Delaunay tessellation of $(S,x)$, in the sense of Lemma \ref{compatible}.  There are three advantages to working with the Delaunay tessellation.  First, every geodesic arc of length $2\mathit{injrad}_x(S)$ based at $x$ is a Delaunay edge, as we prove in Lemma \ref{injrad edge}.  Second, by construction Delaunay cells are cyclic or horocyclic; that is, inscribed in metric circles or horocycles, respectively.  In \cite{DeB_cyclic_geom} there are calculus formulas describing the derivative of area with respect to side length for such polygons.

Finally, the Delaunay tessellation of $\cals$ contains the geometric dual to the \textit{Voronoi tessellation} of $\cals$, which has two-cells of the form
$$V_{\bx} = \{y\in\mathbb{H}^2\,|\,\dist(y,\bx)\leq \dist(y,\bx')\ \forall\ \bx'\in\cals \}, $$
for each $\bx\in\cals$.  See eg.~\cite[\S 5]{DeB_Delaunay}. Its edges are intersections $V_{\bx}\cap V_{\bx'}$ containing at least two points.  The \textit{geometric dual} to any such edge is the geodesic arc joining $\bx$ to $\bx'$.  In Section \ref{the proof, man} we will exploit the geometric duality relation using some results from \cite[\S 2.1]{DeB_Voronoi} that show how the Voronoi tessellation encodes certain extra structure associated to ``non-centered'' Delaunay two-cells.  This helps us overcome the central difficulty in using deformations via triangulations, which is that the area of cyclic polygons is not monotonic in their side lengths.

The first result we will prove here is mostly \cite[Corollary 5.2]{DeB_Voronoi}, which is again the specialization of a result from \cite{DeB_Delaunay}, Corollary 6.27 there.  Theorem 5.1 of \cite{DeB_Voronoi} asserts for a set $\widetilde\cals$ invariant under a lattice $\Gamma$ that a Delaunay cell of $\widetilde\cals$ is inscribed in a horocycle $C$ if and only if its stabilizer in $\Gamma$ is a parabolic subgroup $\Gamma_u$ of $\Gamma$ that fixes the ideal point $C$.  Such cells are the primary concern of this result.

\begin{corollary}\label{horocyclic char}  For a complete, oriented, finite-area hyperbolic surface $F$ with locally isometric universal cover $\pi\co\mathbb{H}^2\to F$, and a finite set $\cals\subset F$, there are finitely many $\pi_1 F$-orbits of Delaunay cells of $\widetilde{\cals}=\pi^{-1}(\cals)$.  The interior of each compact Delaunay cell embeds in $F$ under $\pi$.  For a cell $C_u$ with parabolic stabilizer $\Gamma_u$, $\pi|_{\mathit{int}\,C_u}$ factors through an embedding of $\mathit{int}\,C_u/\Gamma_u$ to a set containing a cusp of $F$.

A fundamental domain in a parabolic-invariant cell $C_u$ for the action of its stabilizer $\Gamma_u$ is a horocyclic ideal polygon whose finite-length edges are edges of $C_u$.\end{corollary}

\begin{proof}  The first part of this result was recorded as Corollary 5.2 of \cite{DeB_Voronoi}.  The second part follows from Lemma 5.7 there.  This result implies that the vertices of $C_u$ can be enumerated as $\{s_i\,|\,i\in\mathbb{Z}\}$ so that $s_i$ and $s_{i+1}$ bound an edge $\gamma_i$ of $C_u$ for each $i$, and $g(s_i) = s_{i+k}$ for some fixed $k\in\mathbb{Z}$, where $g$ is the generator of $\Gamma_u$.  It follows that a fundamental domain for the $\Gamma_u$-action is the non-overlapping union of horocyclic triangles $T_i\cup T_{i+1}\hdots \cup T_{i+k-1}$ defined in Lemma 5.7 for any fixed $i$.  This is a horocyclic ideal $(k+1)$-gon, see \cite[Prop.~3.8]{DeB_cyclic_geom}.\end{proof}

\begin{definition}\label{Delaunay dfn}For a complete, oriented, finite-area hyperbolic surface $S$ and $x\in S$, we will call the \textit{Delaunay tessellation of $(S,x)$} the projection to $S$ of the Delaunay tessellation of $\pi^{-1}(x)$, for some fixed universal cover $\pi\co\mathbb{H}^2\to S$.\end{definition}

\begin{lemma}\label{injrad edge}  For a complete, oriented, finite-area hyperbolic surface $S$ and $x\in S$, every geodesic arc based at $x$ with length $2\mathit{injrad}_x(S)$ is an edge of the Delaunay tessellation of $(S,x)$.  In particular, the injectivity radius of $S$ at $x$ is half the minimum edge length of the Delaunay tessellation.\end{lemma}

\begin{proof}  Every Delaunay edge of $(S,x)$ is a non-constant geodesic arc with both endpoints at $x$, so its length is at least $2\mathit{injrad}_x(S)$.  For a closed geodesic arc $\gamma$ of length $2\mathit{injrad}_x(S)$ based at $x$, let $\tilde{\gamma}$ be a lift of $\gamma$ to $\mathbb{H}^2$.  The metric disk $B$ of radius $\mathit{injrad}_x(S)$ centered at the midpoint of $\tilde{\gamma}$ intersects $\pi^{-1}(x)$ in the endpoints of $\tilde{\gamma}$.  Every other point of $B$ has distance less than $2\mathit{injrad}_x(S)$ from the endpoints of $\tilde{\gamma}$, so $B\cap\pi^{-1}(x) = \partial\tilde{\gamma}$.  It follows that $\tilde{\gamma}$ is a Delaunay edge of $\pi^{-1}(x)$, hence that $\gamma$ is a Delaunay edge of $(S,x)$.\end{proof}

\begin{lemma}\label{compatible}  For any complete, oriented, hyperbolic surface $S$ of finite area and $x\in S$, there is a decomposition $\calt$ of $S$ into compact and horocyclic ideal triangles that is \mbox{\rm compatible} with the Delaunay tessellation of $(S,x)$ in the sense that its vertex set is $\{x\}$ and each edge of the Delaunay tessellation is an edge of $\calt$.\end{lemma}

\begin{proof}$\calt$ is compatible with the Delaunay tessellation if its faces are obtained by subdividing Delaunay two-cells into triangles.  This can be done for instance by dividing each compact two-cell by diagonals from a fixed vertex, and each horocyclic two-cell into horocyclic ideal triangles.  On a horocyclic cell $C_u$,  the latter operation joins each vertex of the fundamental domain for $\Gamma_u$ of Corollary \ref{horocyclic char} to the ideal point of its circumscribed horocycle.\end{proof}

%%%%%%%%%%%%%%%%%
\section{Increasing injectivity radius}\label{the proof, man}
%%%%%%%%%%%%%%%%%

The goal of this section is to prove the main Theorem \ref{only max}.  We will do this in two steps.  The first, Proposition \ref{longer than} below, reduces to the case that all compact Delaunay edges have equal length.  We then prove the Theorem by addressing the case that all Delaunay edge lengths are equal but there is a complicated Delaunay cell $C$.

\begin{proposition}\label{longer than}  For a complete, oriented hyperbolic surface $S$ of finite area and $x\in S$ such that $\maxi(S)=\mathit{injrad}_x(S)$, if the Delaunay tessellation of $(S,x)$ has an edge of length greater than $2\mathit{injrad}_x(S)$ then $S$ is not a local maximum of $\maxi$ on $\mathfrak{T}_{g,n}$.  

In fact, there is a continuous map $t\mapsto S_t\in\mathfrak{T}_{g,n}$ on $(-\epsilon,\epsilon)$ for some $\epsilon>0$, and $x_t\in S_t$ for each $t$, such that $S_0 = S$, $x_0 = x$, and $\frac{d}{dt}\mathit{injrad}_{x_t}(S_t) = \frac{1}{2}$.\end{proposition}

\begin{proof} Let $\calt$ be a triangulation compatible with the Delaunay tessellation of $(S,x)$.  Enumerate the edges of $\calt$ as $\gamma_1,\hdots,\gamma_l$ so that the Delaunay edges consist of those with $j\leq n$ for some $n\leq l$, and $\gamma_j$ has length $2\mathit{injrad}_x(S)$ if and only if $j\leq m$ for some $m<n$, and let $\bd = (d_0,\hdots d_l)$ be the collection of edge lengths.  We will produce a smooth map $t\mapsto \bd(t) = (d_1(t),\hdots,d_n(t))\in\mathfrak{D}_D(S,\calt)$ on some interval $(-\epsilon,\epsilon)$, with $\bd(0) = \bd$, by prescribing the $d_j(t)$ as follows: take $d_j(t) = d_j+t$ for all $j\leq m$ and leave all other edge lengths $d_j(t)$ constant except for $d_n(t)$, which is determined by the ODE $\frac{d}{dt} A_x(\bd(t)) = 0$.

Here $A_x$ is from Definition \ref{triangle def}.  If there is a smooth solution $d_n(t)$ then $A_x(\bd(t)) \equiv 2\pi$ since $\bd$ is the edge length collection of the triangulated hyperbolic surface $(S,\calt)$.  It will then follow from Proposition \ref{parametrize} that $S_t \doteq S(\bd(t))$ is a deformation of $S$ in $\mathfrak{T}_{g,n}$, and from Proposition \ref{Delaunay control} that $\frac{d}{dt} \mathit{injrad}_{x_t}S_t = \frac{1}{2}$, where $x_t$ is the vertex of $S(\bd(t))$.  To show that $\frac{d}{dt} A_x(\bd(t)) = 0$ has a smooth solution we rearrange it using the chain rule and our stipulations on the $d_j(t)$, yielding:\begin{align*}
  0 = & \sum_{j=1}^m \frac{\partial}{\partial d_j} \left(D_0(T_{i_j^+}(\bd(t))) + D_0(T_{i_j^-}(\bd(t))) \right) \\ &\qquad + d_{n}'(t)\cdot\frac{\partial}{\partial d_{n}} \left(D_0(T_{i_{n}^+}(\bd(t))) + D_0(T_{i_{n}^-}(\bd(t)))\right).\end{align*}
Here for each $j$, $T_{i_j^+}$ and $T_{i_j^-}$ are the triangles containing the edge $\gamma_j$; for a triangle $T_i$ with edges $\gamma_{j_1}$, $\gamma_{j_2}$, $\gamma_{j_3}$ we refer by $T_i(\bd(t))$ to the triple $(d_{j_1}(t),d_{j_2}(t),d_{j_3}(t))$ of changing edge lengths; and $D_0(a,b,c)$ records the area of the triangle with edge lengths $a$, $b$ and $c$.  In \cite{DeB_cyclic_geom} we gave formulas for the partial derivatives of $D_0$ with respect to $a$, $b$ and $c$.

If the coefficient $\frac{\partial}{\partial d_{n}} \left(D_0(T_{i_{n}^+}(\bd(t))) + D_0(T_{i_{n}^-}(\bd(t)))\right)$ is non-zero then we can solve for $d_{n}'(t)$, yielding a first-order ODE in $d_{n}(t)$.  We claim this holds at $t=0$, ie for the $T_{i_{n}^{\pm}}(\bd)$, and therefore at all possible values of $\bd(t)$ near $\bd$.  Given the claim, Picard's theorem on the existence of solutions to first-order ODE implies there is a smooth solution $d_{n}(t)$ for small $t$ (note that smoothness of $D_0$ is proven in \cite{DeB_cyclic_geom}).  We will apply results from Section 2 of \cite{DeB_Voronoi}, together with \cite[Proposition 2.3]{DeB_cyclic_geom}, to prove the claim.  

There are two cases, divided by the qualitative nature of the Delaunay cells $C_{i_n}^{\pm}$ of $(S,x)$ containing the triangles $T_{i_n}^{\pm}$.  In the first case one of the $C_{i_{n}^{\pm}}$, say $C_{i_n^-}$, is compact and therefore cyclic but not centered, and $\gamma_n$ is its longest side.  Here a cyclic polygon is \textit{centered} if its interior contains the center of its circumcircle.  
The longest side of a non-centered cyclic polygon separates its interior from the center of its circumcircle \cite[Prop.~2.2]{DeB_cyclic_geom}, so since $T_{i_n}^-$ is contained in $C_{i_n}^-$ it is also non-centered with longest side $\gamma_n$.

In this case Lemma 2.5 of \cite{DeB_Voronoi} asserts that the dual Voronoi vertex $v$ to $C_{i_{n}^-}$ is the initial vertex of a non-centered Voronoi edge $e$ geometrically dual to $\gamma_n$.  If $e$ is compact then its terminal vertex is the geometric dual to $C_{i_{n}^+}$, so by the same result it is not also the case that $C_{i_{n}^+}$ is non-centered with longest edge $\gamma_{j_0}$.  This is therefore also not the case for $T_{i_n^-}$, so be \cite[Proposition 2.3]{DeB_cyclic_geom} the coefficient of $d_{n}'(t)$ at $t=0$ is:
\begin{align}\label{centered non-centered} %\frac{\partial}{\partial d_{j_0}} \left(D_0(C_{i_{j_0}^+}(\bd)) + D_0(C_{i_{j_0}^-}(\bd))\right) = 
\sqrt{\frac{1}{\cosh^2(d_{n}/2)} - \frac{1}{\cosh^2 J(T_{i_{n}^+}(\bd))}} - \sqrt{\frac{1}{\cosh^2(d_{n}/2)} - \frac{1}{\cosh^2 J(T_{i_{n}^-}(\bd))}} \end{align}
Above, $J(T_{i_{n}^+}(\bd))$ is the circumcircle radius of $C_{i_{n}^+}$, and therefore also of $T_{i_n^+}$, and likewise for $J(T_{i_{n}^-}(\bd))$.  Lemma 2.3 of \cite{DeB_Voronoi} implies that the former is greater than the latter, and it follows in this sub-case that the coefficient of $d_{n}'(t)$ is greater than $0$ at $t=0$.

It is also possible in this case that the Voronoi edge $e$ geometrically dual to $\gamma_{n}$ is noncompact.  Then arguing as in the proof of \cite[Lemma 5.8]{DeB_Voronoi} establishes that  $C_{i_{n}^+}$ is also non-compact: for the universal cover $\pi\co\mathbb{H}^2\to S$, if $\cals = \pi^{-1}(x)\subset \mathbb{H}^2$ and $\tilde{e}$ is a lift of $e$, Lemma 1.9 of \cite{DeB_Voronoi} asserts that its ideal endpoint $v_{\infty}$ is the ideal point of a horocycle $S$ with the property that the horoball $B$ bounded by $S$ satisfies $B\cap\cals = S\cap\cals$, and $S$ contains the endpoints of the geometric dual $\gamma$ to $\tilde{e}$.  By the empty circumcircles condition, the convex hull of $B\cap\cals$ is a non-compact Delaunay two-cell $\tilde{C}$ containing $\gamma$, which is a lift of $\gamma_{n}$ since $\tilde{e}$ is a lift of $e$.  Hence $\tilde{C}$ projects to a non-compact two-cell containing $\gamma_{n}$, necessarily $C_{i_{n}^+}$.

Corollary \ref{horocyclic char} implies that on the interior of $\tilde{C}$ the projection to $C_{i_{n}^+}$ factors through an embedding of $\mathit{int}(\tilde{C})/\Gamma$, where $\Gamma$ is the stabilizer of $v_{\infty}$ in $\pi_1 S$.  We may assume that the triangulation of $C_{i_n^+}$ has been obtained by dividing $\tilde{C}$ into triangles with geodesic rays joining its vertices to $v_{\infty}$, then projecting, so in particular $T_{i_n^+}$ is the projection of a horocyclic ideal triangle with compact side of length $d_n$.  From the second equation of Proposition 3.7 of \cite{DeB_cyclic_geom} we therefore obtain:\begin{align}\label{horocyclic non-centered}
  \frac{\partial}{\partial d_{n}} \left(D_0(T_{i_{n}^+}(\bd)) + D_0(T_{i_{n}^-}(\bd))\right) =  \frac{1}{\cosh(d_{n}/2)} - \sqrt{\frac{1}{\cosh^2(d_{n}/2)} - \frac{1}{\cosh^2 J(T_{i_{n}^-}(\bd))}} \end{align}
Again this is positive, and the claim follows in this case.

The second case of the claim is when neither of $C_{i_{n}^{\pm}}$ is non-centered with longest edge $\gamma_{n}$, whence the same holds for the $T_{i_n^{\pm}}$.  In this case both terms of the coefficient of $d_{n}'(t)$ are positive, by Propositions 2.3 or 3.7 of \cite{DeB_cyclic_geom}, applied as above.
\end{proof}

\begin{proof}[Proof of Theorem \ref{only max}]  Let $(S,x)$ be a pointed surface whose Delaunay tessellation is not of the form described in the Theorem.  The goal is to show that there is a deformation of $(S,x)$ that increases injectivity radius at $x$.  We will assume we are in the case not covered by Proposition \ref{longer than}: all Delaunay edges of $(S,x)$ have length $2\mathit{injrad}_x(S)$, and there is a Delaunay two-cell $C$ which is compact and not a triangle, or non-compact and not a monogon.

Let $\calt$ be a triangulation that is compatible with the Delaunay tessellation of $(S,x)$, with an edge $\gamma_0$ that is a diagonal of $C$ with a compact triangle $T_1$ of $\calt$ on one side and the remainder of $C$ on the other.  If $C$ is non-compact we accomplish this as follows: for a locally isometric universal cover $\pi\co\mathbb{H}^2\to S$ and a horocyclic two-cell $\widetilde{C}$ of the Delaunay tessellation of $\pi^{-1}(x)$ projecting to $C$, let $\tilde{\gamma}_0$ join vertices of $\widetilde{C}$ separated by exactly one other vertex on the horocycle in which it is inscribed.  Then the compact subregion $\widetilde{T}_1$ of $\widetilde{C}$ that it bounds is a triangle, hence so is its projection $T_1$.  In this non-compact case we divide the remainder of $C$, and all other horocyclic Delaunay cells, into horocyclic ideal triangles as previously.

Enumerate the edges of $\calt$ as $\gamma_0,\hdots,\gamma_l$ so that the Delaunay edges are $\gamma_1,\hdots,\gamma_n$ for some $n\leq l$.  Let $d_j$ be the length of $\gamma_j$ for each $j$, and note that by hypothesis $\gamma_j$ has length $d\doteq2\mathit{injrad}_x(S)$ for $0<j\leq n$.  Now let $\bd = (d_0,\hdots,d_l)$, and prescribe $\bd(t) = (d_0(t),\hdots,d_l(t))$ with $\bd(0) = \bd$ as follows: $d_0(t) = d_0-t$; $d_j(t) \equiv d_j$ for $j>n$; and for $0<j\leq n$, $d_j(t) = d(t)$ is determined by the differential equation $\frac{d}{dt}A_x(\bd(t)) = 0$.

Here as in the proof of Proposition \ref{longer than}, $A_x$ is the angle sum function from Definition \ref{triangle def}, and for a smooth solution $\bd(t)$ we have $A_x(\bd(t))\equiv 2\pi$ since $\bd$ is the edge length collection of the triangulated hyperbolic surface $(S,\calt)$.  It will then follow from Proposition \ref{parametrize} that $S(\bd(t))$ is a deformation of $S$ in $\mathfrak{T}_{g,n}$.  And if $d(t)$ increases with $t$, then by Proposition \ref{Delaunay control}, the injectivity radius of $S(\bd(t))$ at its vertex will as well.  We will show this below.

As in the proof of Proposition \ref{longer than} we rewrite the equation $\frac{d}{dt}A_x(\bd(t)) = 0$ using the chain rule and our choices for $\bd(t)$:\begin{align}\label{eggman}
  d'(t)\sum_{j=1}^{n} \frac{\partial}{\partial d_j}\left(D_0(T_{i_j^+}(\bd(t)))+D_0(T_{i_j^-}(\bd(t))) \right) - \frac{\partial}{\partial d_0}\left(D_0(T_0(\bd(t)))+D_0(T_1(\bd(t)))\right) =0 \end{align}
Again as in Proposition \ref{longer than}, for each $j>0$ the $T_{i_j^{\pm}}$ are the triangles containing the edge $\gamma_j$.  Here $T_0$ and $T_1$ are the triangles containing $\gamma_0$, and by construction, $T_1$ is compact.  In all cases if $T_{i_j^{\pm}}$ has edges $\gamma_{j_1}$, $\gamma_{j_2}$, $\gamma_{j_3}$ then $T_{i_j^{\pm}}(\bd(t))$ refers to the collection $(d_{j_1}(t),d_{j_2}(t),d_{j_3}(t))$ of changing edge lengths.  We claim first that all coefficients above are smooth, and that the coefficient of $d'(t)$ is positive.

To the latter point, recall that since $\calt$ is compatible with the Delaunay tessellation of $(S,x)$, each $T_{i_j^{\pm}}$ is contained in a Delaunay cell $C_{i_j^{\pm}}$.  If $C_{i_j^{\pm}}$ is compact it is centered, being equilateral, so since $T_{i_j^{\pm}}$ has the same circumcircle it is either centered or one of its edges is a diagonal that separates it from the circumcircle center.  In neither of these cases is it non-centered with longest edge $\gamma_j$, so by Proposition 2.3 of \cite{DeB_cyclic_geom} its contribution to the coefficient of $d'(t)$ is positive.  If the Delaunay cell $C_{i_j^{\pm}}$ containing $T_{i_j^{\pm}}$ is horocyclic, and hence $T_{i_j^{\pm}}$ is a horocyclic ideal triangle by construction, then this follows from Proposition 3.7 of \cite{DeB_cyclic_geom}.

Smoothness of the coefficients of (\ref{eggman}) follows from results of \cite{DeB_cyclic_geom}.  In particular, Proposition 2.3 there asserts that $D_0$ is smooth on the set $\calAC_3\subset(0,\infty)^3$ parametrizing cyclic triangles.  This applies to each $T_{i_j^{\pm}}$ contained in a compact Delaunay cell.  Each one contained in a horocyclic cell, except possibly $T_1$, is a horocyclic ideal triangle by construction, and smoothness follows by \cite[Prop.~3.7]{DeB_cyclic_geom}.  If $T_1$ is in a horocyclic Delaunay cell then its side-length collection $T_1(\bd)$ lies in the set $\calHC_3$ of \cite[Corollary 3.5]{DeB_cyclic_geom}, parametrizing compact ``horocyclic'' triangles.  $\calHC_3$ has codimension one in $(0,\infty)^3$.  It bounds the set $\calc_3$ parametrizing cyclic triangles on one side, and the set $\cale_3$ parametrizing ``equidistant'' triangles on the other (see \cite[\S 4]{DeB_cyclic_geom}; in particular Cor.~4.6 there).  

For arbitrary $n\geq 3$, the versions of $D_0$ that record areas of horocyclic and equidistant $n$-gons are respectively defined in Propositions 3.7 and 4.9 of \cite{DeB_cyclic_geom}.  We proved there that the various definitions of $D_0$ determine a continuous function on $\calAC_n\cup\calHC_n\cup\cale_n$, but we did not address smoothness on $\calHC_n$.  However since $D_0$ measures area, for $n=3$ it agrees everywhere with the smooth function $A$ of \cite[Lemma 1.16]{DeB_cyclic_geom}.  Therefore since $T_1$ is a triangle, the coefficient function $D_0(T_1(\bd(t)))$ of (\ref{eggman}) is smooth.

Since the coefficient of $d'(t)$ in (\ref{eggman}) is positive at $\bd$ and all coefficients are smooth, there is a smooth solution $d(t)$ near $t=0$. The sign of $d'(t)$ is determined by the sign of $\frac{\partial}{\partial d_0}\left(D_0(T_0(\bd(t))+D_0(T_1(\bd(t)))\right)$.

\begin{claim}\label{increasing} For small $t>0$, $\frac{\partial}{\partial d_0}\left(D_0(T_0(\bd(t))+D_0(T_1(\bd(t)))\right)>0$, hence $d'(t)>0$.\end{claim}

\begin{proof}[Proof of claim] First suppose $C$ is non-compact.  Then $d'(0)=0$, since at time $0$ we have:
$$ \frac{\partial}{\partial d_0} D_0(T_0(\bd)) = \frac{1}{\cosh(d_0/2)} = -\frac{\partial}{\partial d_0} D_0(T_1(\bd)); $$
The computation here for $T_0$ is obtained by taking a derivative with respect to $d_0$ of the second formula of \cite[Proposition 3.7]{DeB_cyclic_geom}.  For $T_1$ it follows similarly from the first formula there, noting that by construction $d_0$ is the largest side length of $T_1$.  

This requires some comment since the formula in question applies only to points of $\calHC_n$, which as we pointed out above is codimension-one in $(0,\infty)^n$.  But since we have chosen $\gamma_0$ so that $n=3$, as pointed out above $D_0$ is smooth on a neighborhood of the side length collection $T_1(\bd) = (d_0,d,d)$ of $T_1$, and its partial derivative with respect to $d_0$ at this point is a limit of $\frac{\partial D_0}{\partial d_0}(\bd_n)$ for a sequence $\{\bd_n\}\in\calAC_3$ approaching $T_1(\bd)$.  Noting that all but finitely many $\bd_n$ are in $\calAC_3-\calc_3$, by Corollary 3.5 of \cite{DeB_cyclic_geom}, and the circumcircle radius $J(\bd_n)\to\infty$ as $n\to\infty$, by Proposition 3.6 there, the given formula is a limit of the one given by Proposition 2.3 there.

By the above we have that $\frac{d}{dt}T_1(\bd(0)) = (-1,0,0)$.  Near $(d_0,d,d)$, $\calHC_3$ is characterized as a graph $\{h_0(x,y),x,y)\}$ by Corollary 3.5 of \cite{DeB_cyclic_geom}, for $h_0$ as defined there, and $\calAC_3$ is characterized as $\{(x,y,z)\,|\, x<h_0(y,z)\}$; compare with \cite[Corollary 1.10]{DeB_cyclic_geom}.  Thus this vector points into $\calAC_3$, so $T_1(\bd(t))\in\calAC_3$ for all small-enough $t>0$.  For all such $t$ it follows that $\frac{\partial}{\partial d_0}\left(D_0(T_0(\bd(t))+D_0(T_1(\bd(t)))\right)$ is given by the formula of (\ref{horocyclic non-centered}), with $d_0$ replacing $d_{n}$ and $T_0(\bd(t))$ replacing $T_{i_{n}^-}(\bd_0)$ there.  This quantity is positive, therefore so is $d'(t)$, and the claim holds if $C$ is non-compact.

We now address the case that $C$ is compact.  First suppose that $C$ is a quadrilateral.  By hypothesis all its edge lengths are equal to $d=2\mathit{injrad}_x(S)$, so since it is cyclic and therefore uniquely determined by its edge length collection it is fully symmetric.  In particular, each diagonal of $C$ is a diameter of its circumcircle, so $d_0 = 2J(d_0,d,d)$, where $J\co\calAC_n\to(0,\infty)$ records circumcircle radius of cyclic polygons; see Proposition 1.14 of \cite{DeB_cyclic_geom}.  Plugging this into Proposition 2.3 there gives $d'(0) = 0$ again.

Again in this case we have $\frac{d}{dt}T_1(\bd) = (-1,0,0) = \frac{d}{dt}T_0(\bd)$.  In this case the edge length collections of $T_0$ and $T_1$ lie in the set $\calBC_3$ parametrizing \textit{semicyclic} triangles, cyclic triangles with one side a diameter of their circumcircles.  This is a codimension-one submanifold of $(0,\infty)^3$ which is the frontier of $\calc_3$, the open set parametrizing centered triangles, in $\calAC_3$; see \cite[Proposition 1.12]{DeB_cyclic_geom}.  The vector $(-1,0,0)$ points into $\calc_3$ at $T_1(\bd)$, since near here $\calBC_3$ is a graph $\{(b_0(x,y),x,y)\}$ (see \cite[Prop.~1.12]{DeB_cyclic_geom}) and $\calc_3 = \{(z,x,y)\,|\, z<b_0(x,y)\}$ (compare \cite[Prop.~1.11]{DeB_cyclic_geom}).  Therefore $T_1(\bd(t)) = T_0(\bd(t))\in\calc_3$ for all small $t>0$, and it follows from Proposition 2.3 of \cite{DeB_cyclic_geom} that $d'(t)>0$ for such $t$.

If $C$ is not a quadrilateral then we may choose $\gamma_0$ and $T_1$ so that the circumcircle center of $C$ lies on the opposite side of $\gamma_0$ from $T_1$.  Then $T_1(\bd)\in\calAC_{3}-(\calc_{3}\cup\calBC_{3})$ has largest entry $d_0$.  On the other hand either $T_0(\bd)\in\calc_{3}$, i.e.~$T_0$ is centered, or $T_0$ is not centered and $\gamma_0$ is not its longest side.  The condition on $T_1(\bd)$, being open, holds for $T_1(\bd(t))$ for all $t$ near $0$.  Similarly, if $T_0(\bd)\in\calc_3$ then this also holds for $T_0(\bd(t))$, or if $d_0$ is not the largest entry of $T_0(\bd)$ then $d_0(t)$ is not the largest entry of $T_0(\bd(t))$, for all $t$ near $0$.  Proposition 2.3 of \cite{DeB_cyclic_geom} thus implies that $\frac{\partial}{\partial d_0}\left(D_0(T_0(\bd(t))+D_0(T_1(\bd(t)))\right)$ is given for all such $t$ by the formula (\ref{centered non-centered}), with $d_0$ replacing $d_{n}$, $T_0(\bd(t))$ replacing $T_{i_{n}^+}(\bd)$, and $T_1(\bd(t))$ replacing $T_{i_{n}^-}(\bd)$.

We have $d'(0) = 0$ by (\ref{centered non-centered}), since $T_0$ and $T_1$ are both inscribed in the circumcircle of $C$.  For $t\ne 0$, if $J(T_0(\bd(t))) > J(T_1(\bd(t)))$ then $d'(t)>0$, again by (\ref{centered non-centered}).  Applying Proposition 1.14 of \cite{DeB_cyclic_geom}, we obtain either
$$0<\frac{\partial}{\partial d_0} J(T_0(\bd)) < 1/2 <  \frac{\partial}{\partial d_0} J(T_1(\bd)), $$
if $T_0$ is centered (i.e. $T_0(\bd)\in\calc_3$), or $\frac{\partial}{\partial d_0} J(T_0(\bd)) < 0$ if not.  Since $d'(0)=0$, the chain rule implies that $\frac{d}{dt} J(T_0(\bd)) = -\frac{\partial}{\partial d_0} J(T_0(\bd))$, and similarly for $\frac{d}{dt} J(T_1(\bd))$.  Thus for $t>0$, $J(T_0(\bd))>J(T_1(\bd))$ so $d'(t)>0$, and the claim is proved in all cases.\end{proof}

Lemma \ref{hyperbolize} now implies that $\bd(t)$ determines a path $\left(S(\bd(t)),\calt(\bd(t))\right)$ of triangulated hyperbolic surfaces, which Proposition \ref{parametrize} implies is continuous in $\mathfrak{T}_{g,n}$.  By Proposition \ref{Delaunay control} and our construction, $S_t$ has injectivity radius $d(t)/2$ at the vertex of $\calt_t$, so since $d$ increases with $t$ the result holds.\end{proof}

\bibliographystyle{plain}
\bibliography{local_max}

\end{document}